\documentclass[a4paper,11pt]{article}
\usepackage[latin1]{inputenc}
\usepackage[T1]{fontenc}
\usepackage[normalem]{ulem}
\usepackage[english]{babel}
\usepackage{verbatim}
\usepackage{graphicx}
\usepackage{amsmath,amssymb,amsfonts,amsrefs,amsthm}
\usepackage{rotating}
\usepackage{enumerate}

\newtheorem{theorem}{Theorem}[section]
\newtheorem*{maintheorem}{Main Theorem}
\newtheorem{lemma}[theorem]{Lemma}
\newtheorem{prop}[theorem]{Proposition}

\theoremstyle{definition}
\newtheorem{defn}[theorem]{Definition}
\newtheorem{rem}[theorem]{Note}
\newtheorem{ex}[theorem]{Example}

\input xy
\xyoption{all}
\newcommand{\sbt}{\,\begin{picture}(-1,1)(0.5,-1)\circle*{1.8}\end{picture}\hspace{.05cm}}

\DeclareMathOperator{\pr}{pr}

\DeclareMathOperator{\id}{id}
\DeclareMathOperator{\im}{Im}

\DeclareMathOperator{\map}{Map}

\DeclareMathOperator{\Top}{Top}
\DeclareMathOperator{\Ob}{Ob}
\DeclareMathOperator{\Mor}{Mor}
\DeclareMathOperator{\emb}{Emb}
\DeclareMathOperator{\colim}{colim}
\DeclareMathOperator{\fr}{Fr}

\DeclareMathOperator{\C}{\mathcal{C}_{\Psi}(M)}
\DeclareMathOperator{\D}{D_{\Psi}(M)}

\DeclareMathOperator{\res}{res}
\DeclareMathOperator{\sing}{Sing}

\begin{document}

\begin{center}\LARGE{A relative h-principle via cobordism-like categories}\end{center}

\begin{center}$\mbox{\large{Emanuele Dotto}}^{{\ast}}$\let\thefootnote\relax\footnotetext{$^\ast$Partially supported by
ERC Adv.Grant No.228082. and by DNRF Centre for Symmetry and Deformation
}\end{center}
\vspace{.3cm}
\begin{quote}
\textsc{Abstract}. We prove an h-principle with boundary condition for a certain class of sheaves $\Psi\colon\emb_{d}^{op}\longrightarrow \Top$. The techniques used in the proof come from the study of the homotopy type of cobordism categories, and they are of simplicial and categorical nature. Applying the main result of this paper we recover the homotopy equivalence $B\mathcal{C}_{k,d}\simeq\Omega^{d-1}Th(\gamma_{k,d}^\bot)$ of \cite{soren} and \cite{GMTW}.
\end{quote}
\vspace{-.5cm}
\tableofcontents

\section*{Introduction}
Given a continuous sheaf $\Psi\colon\emb_{d}^{op}\longrightarrow \Top$ on the category of smooth $d$-dimensional manifolds without boundary with embeddings as morphisms, one can define a new sheaf $\Psi^\ast\colon\emb_{d}^{op}\longrightarrow \Top$ together with an h-principle (or scanning) map $h\colon\Psi\longrightarrow\Psi^{\ast}$. Gromov shows in \cite{gromov} that if $\Psi$ is "microflexible" for an open manifold $M$, the map $h\colon\Psi(M)\longrightarrow\Psi^{\ast}(M)$ is a weak homotopy equivalence. 

In this paper we prove a relative result. One can define $\Psi$ for a manifold with boundary $M$ by gluing an external collar, obtaining a restriction map $\Psi(M)\longrightarrow\Psi(\partial M\times \mathbb{R})$. Given an element $g_0\in\Psi(\partial M\times\mathbb{R})$ consider the space $\Psi(M;g_0)$ of elements of $\Psi(M)$ that restrict to $g_0$ in a neighborhood of $\partial M\times [0,\infty)$. We find conditions on $\Psi$ so that the restriction of the h-principle map $h\colon\Psi(M;g_0)\longrightarrow\Psi^\ast(M;h(g_0))$ is a weak equivalence.

Our approach to this problem is motivated by the study of the weak homotopy type of cobordism categories. Let $\mathcal{C}_{k,d}$ be the cobordism category of \cite{GMTW}, having $k$-dimensional compact submanifolds of $\mathbb{R}^{d-1}\times [a,b]$ as morphisms. The Madsen-Weiss theorem of \cite{soren},\cite{GMTW} and \cite{MadsenWeiss} shows the existence of a weak equivalence
\[B\mathcal{C}_{k,d}\simeq\Omega^{d-1}Th(\gamma_{k,d}^\bot)\]
where $\gamma_{k,d}^\bot=\{(V,v)\in\mathcal{G}_{k,d}\times\mathbb{R}^{d}|v\bot V\}$ is the total space of the complement of the tautological bundle over the Grassmanians $\mathcal{G}_{k,d}$ of $k$-vector subspaces of $\mathbb{R}^{d}$. This equivalence can be rephrased as a relative h-principle for the sheaf $\Psi_k\colon\emb_{d}^{op}\longrightarrow \Top$ defined by the set
\[\Psi_k(M)=\{W\subset M| W \mbox{ is a $k$-submanifold, closed as a subset}\}\]
suitably topologized (cf. \cite{soren},\cite{osmonoids}). The boundary condition needed for the Madsen-Weiss theorem is the empty submanifold $\emptyset\in\Psi_k(\partial M\times \mathbb{R})$.

In this paper we modify the proof of \cite{david} of the Madsen-Weiss theorem using a Quillen B argument. Then we reformulate the structural properties of $\Psi_k$ used in the proof to obtain conditions on a general sheaf $\Psi$, and we obtain the following.

\begin{maintheorem}
Suppose that
\begin{enumerate}
\item The h-principle maps $\Psi(M)\longrightarrow \Psi^\ast(M)$ and $\Psi(\partial M\times\mathbb{R})\longrightarrow \Psi^\ast(\partial M\times\mathbb{R})$ are weak equivalences,
\item The orthogonal elements are almost open in $\Psi(M)$ (see \ref{transverse} and \ref{almostopen}),
\item $\Psi$ is group like at $M$ (see \ref{gplike}),
\item $\Psi$ is damping at $M$ (see \ref{damping}).
\end{enumerate}
Then for any $g_0\in \Psi(\partial M\times\mathbb{R})$ orthogonal to $\partial M$, the relative h-principle map $h\colon\Psi(M;g_0)\longrightarrow \Psi^\ast(M;h(g_{0}))$ is a weak equivalence.
\end{maintheorem}

For the sheaf $\Psi_k$ of the Madsen-Weiss theorem, the properties above reduce essentially to the following. Orthogonal elements of $\Psi_k(\partial M\times\mathbb{R})$ are submanifolds that intersect $\partial M\times 0$ orthogonally in the classical sense. These elements are "almost open", in the sense that the inclusion into elements that intersect $\partial M\times 0$ transversally is an equivalence, and transversality is an open condition. The group-like condition for $\Psi_k$ is the existence of a "dual manifold" for every submanifold $N\subset \partial M\times[0,1]$ that intersects the boundary orthogonally. It is the submanifold of $\partial M\times[1,2]$ defined by $2-N$. The damping condition is a formal condition needed for proving that structural maps for a Quillen theorem B are weak equivalences. It holds for $\Psi_k$ by the smooth approximation theorem of \cite{osmonoids}.

In order to prove the main theorem we build a model for the restriction map $\Psi(M)\longrightarrow\Psi(\partial M)$ using a functor of simplicial categories $\C\longrightarrow\partial\C$. Here $\C$ and $\partial\C$ are defined imitating the construction of cobordism categories. The conditions above allow us to build weak equivalences
\[\xymatrix{B\C \ \simeq \ \Psi(M)\ar@<6ex>[d]\ar@<-6ex>[d]\\
B\partial\C\simeq\Psi(\partial M)
}\]
analogous to \cite{david} and \cite{osmonoids}. Then 
we identify the homotopy fiber of the left-hand map using Quillen theorem B' of \cite{wald}.

In the last section of the paper we show that the sheaf of submanifolds $\Psi_k$ of \cite{soren} and \cite{osmonoids} satisfies the conditions above, giving a "Quillen B proof" of the Madsen-Weiss theorem. We can moreover replace the empty boundary condition with any manifold $g$ that is orthogonal near the boundary of $D^{d-1}\times\mathbb{R}$. This gives a description of the homotopy type of the cobordism category of bordisms that agree with $g$ around $S^{d-2}\times \mathbb{R}$.

\subsection*{Acknowledgment} I am extremely thankful to David Ayala for supervising me closely during this whole project. Without his inspiring ideas and helpful suggestions this paper would not exist.

\section{Relative h-principles and strategy}\label{sec1}

Let $\emb_d$ be the category with objects boundaryless $d$-dimensional manifolds and embeddings as morphisms. This category is enriched in topological spaces endowing the set of morphisms between two manifolds with the $C^\infty$-topology.

We will consider continuous sheaves $\Psi\colon {\emb_d}^{op}\longrightarrow \Top$, where $\emb_d$ has the standard Grothendieck topology. Given such a sheaf, we extend it to ${\emb_{d-1}}^{op}$ by
\[\Psi(N):=\Psi(N\times\mathbb{R})\]
We extend it to $d$-dimensional manifolds with boundary as follows. Let $M$ be a $d$-manifold with boundary, and $e\colon \partial M\times(-\infty,0]\stackrel{\cong}{\longrightarrow}U_e\subset M$ be a collar of $\partial M$. Here $U_e$ is an open neighborhood of $\partial M$, and $e$ identifies $\partial M\times 0$ with $\partial M$. The collar induces a smooth structure on
\[M_{<\infty}:=M\coprod_{\partial M}\partial M\times[0,\infty)\]
and we define
\[\Psi(M):=\Psi(M_{<\infty})\]
Different choices of collars define diffeomorphic manifolds $M_{<\infty}$, and therefore homeomorphic $\Psi(M)$. The collar $e$ also defines a restriction map 
\[\res\colon\Psi(M)\longrightarrow \Psi(\partial M)\]
that sends an element $G\in \Psi(M_{<\infty})$ to 
\[\res(G)=G|_{\partial M}=\widetilde{e}^\ast(G|_{U_e\coprod_{\partial M}\partial M\times[0,\infty)})\]
where $\widetilde{e}\colon\partial M\times \mathbb{R}\stackrel{\cong}{\longrightarrow} U_e\coprod_{\partial M}\partial M\times[0,\infty)$ is the extension of $e$ by the identity on $\partial M\times[0,\infty)$.
For different collars we get the same restriction map under the homeomorphisms above.

There is a new sheaf $\Psi^\ast\colon {\emb_d}^{op}\longrightarrow \Top$ defined by
\[\Psi^\ast(M)=\Gamma(\fr(M)\times_{GL_d}\Psi(\mathbb{R}^d)\longrightarrow M)\]
where $\fr(M)$ is the principal $GL_d$-bundle of framings of $TM$, and $\Gamma$ denotes the space of smooth sections of a bundle.
We extend $\Psi^\ast$ in the same way to lower dimensional manifolds and to collared manifolds with boundary.

The sheaf $\Psi^{\ast}$ comes equipped with a map $\Psi\stackrel{h}{\longrightarrow}\Psi^\ast$ defined as follows. Let $\rho\colon TM\longrightarrow M$ be a fiberwise embedding induced by the exponential map. It can be defined rescaling a ball of $T_xM$ for all $x\in M$ with a radius varying continuously with $x$. We associate to a $G\in \Psi(M)$ the section $s_G\colon M\longrightarrow \fr(M)\times_{GL_d}\Psi(\mathbb{R}^d)$ defined by
\[s_G(x)=[\phi_x,{\phi_x}^\ast\rho_{x}^\ast(G)]\]
for a choice of framing $\phi_x\colon\mathbb{R}^d\stackrel{\cong}{\longrightarrow}T_xM$.

Given a "boundary condition" $g_0\in\Psi(\partial M)$ define
\[\Psi(M;g_0)=\colim_{\epsilon>0}\{G\in\Psi(M):\ G|_{\partial M\times(-\epsilon,\infty)}=g_0|_{\partial M\times(-\epsilon,\infty)}\}\]
the space of elements of $\Psi(M)$ that restrict to $g_0$ in a neighborhood of $\partial M\times [0,\infty)$.
Here we used the notation $G|_{\partial M\times I}:=\res(G)|_{\partial M\times I}$ for an open $I\subset \mathbb{R}$.
The restriction of the h-principle map $h$ induces a map \[\Psi(M;g_0)\longrightarrow\Psi^\ast(M;h(g_0))\]
Our goal is to see when this map is a weak equivalence. We consider the following family of "orthogonal" boundary conditions $g_0\in\Psi(\partial M)$.

\begin{defn}\label{transverse} Let $U\subset M_{<\infty}$ and $I\subset \mathbb{R}$ be open subsets with $\partial M\times I\subset U$. An element $G\in\Psi(U)$ is \textbf{orthogonal to $\partial M$ on $I$}, denoted $G\bot_I$, if for every open intervals $J,J'\subset I$ and any orientation preserving diffeomorphism $\alpha\colon J'\longrightarrow J$
\[G|_{J'}=\alpha^{\ast}(G|_{J})\]
Here $\alpha^\ast\colon \Psi(\partial M\times J)\longrightarrow\Psi(\partial M\times J')$ is
the map induced by $\id\times\alpha\colon \partial M\times I\longrightarrow \partial M\times J$.

Define $\Psi(U)^{\bot_I}\subset\Psi(U)$ to be the subspace of elements that are orthogonal to $\partial M$ on $I$ 
\end{defn}

\begin{theorem}\label{main}
Let $\Psi\colon\emb_{d}^{op}\longrightarrow \Top$ be a continuous sheaf, and $M$ a $d$-dimensional manifold with boundary. Suppose that the following conditions hold:
\begin{enumerate}
\item The h-principle maps $\Psi(M)\longrightarrow \Psi^\ast(M)$ and $\Psi(\partial M)\longrightarrow \Psi^\ast(\partial M)$ are weak equivalences,
\item The orthogonal elements are almost open in $\Psi(M)$ (see \ref{transverse} and \ref{almostopen}),
\item $\Psi$ is group like at $M$ (see \ref{gplike}),
\item $\Psi$ is damping at $M$ (see \ref{damping}).
\end{enumerate}
Then for any $g_0\in \Psi(\partial M)^{\bot_{\mathbb{R}}}$ the relative h-principle map $\Psi(M;g_0)\longrightarrow \Psi^\ast(M;h(g_{0}))$ is a weak equivalence.
\end{theorem}

\begin{rem} By a theorem of Gromov \cite{gromov} the first condition is satisfied if $M$ is an open manifold and the sheaf $\Psi$ is microflexible at $M$ (that is, if the inclusions of compact pairs $K\subset K'\subset M$ induce microfibrations $\Psi(K')\longrightarrow \Psi(K)$ of quasi-topological spaces. See \cite{oscar} for a formulation in terms of lifting properties for topological spaces).
\end{rem}

\begin{proof}
Consider the commutative digram
\[\xymatrix{ \Psi(M;g_0)\ar[d]\ar[r]^-{h} & \Psi^\ast(M;h(g_0))\ar[d]\\
\Psi(M)\ar[d]_{res}\ar[r]^-{h} & \Psi^\ast(M)\ar[d]^{res}\\
\Psi(\partial M)\ar[r]^-{h} & \Psi^\ast(\partial M)
}\]
By the first condition, the homotopy fibers of the restriction maps over the element $g_0$ are weakly equivalent. The right-hand restriction map is a fibration, and therefore its homotopy fiber is equivalent to the preimage of $g_0$, that is equivalent to the space $\Psi^\ast(M;h(g_0))$.

It remains to identify the homotopy fiber of the left-hand restriction map with $\Psi(M;g_0)$. In §\ref{sec2} we define a simplicial functor $\res\colon\C\longrightarrow\partial\C$ and using the second condition we define weak equivalences
\[\xymatrix{B\C \ \simeq \ \Psi(M)\ar@<6ex>[d]\ar@<-6ex>[d]\\
B\partial\C\simeq\Psi(\partial M)
}\]
(see \ref{catmodel}).
In §\ref{sec3} we use the third and the fourth conditions to prove that the simplicial functor $\res\colon\C\longrightarrow\partial\C$ satisfies the hypothesis of a Quillen theorem B for simplicial categories (cf. \ref{firstkind} and \ref{secondkind}). This theorem B gives a description of the homotopy fiber of the realization of our functor as the realization of a certain category. We identify the classifying space of this category with $\Psi(M;g_0)$ in \ref{identifhtpyfib} and \ref{htpyfiber}.
\end{proof}

\begin{ex}\label{sheafpsik}
The main example of a sheaf that satisfies the hypothesis above is $\Psi_k\colon {\emb_d}^{op}\longrightarrow \Top$ from \cite{soren} and \cite{osmonoids}. It associates to a $d$-manifold $M$ the set $\Psi_k(M)$ of closed subsets $N\subset M$ which are smooth $k$-dimensional submanifolds without boundary, suitably topologized.

For an embedding $e\colon M\longrightarrow M'$, the induced map $e^\ast\colon\Psi_k(M')\longrightarrow\Psi_k(M)$ sends a submanifold $N\subset M'$ to $e^{-1}(N)$. The sheaf gluing is given by taking unions.

The condition for a submanifold $G\in\Psi_k(\partial M)$ to be orthogonal to $\partial M$ over $I$ corresponds to orthogonality in the classical sense. That is, $G\bot_I$ exactly when the intersection $G\cap\partial M\times I$ is a product manifold $N\times I$ for some $(k-1)$-dimensional submanifold $N\subset\partial M$.

In section \ref{4aut} we show that $\Psi_k$ satisfies the conditions of \ref{main} above, and we relate theorem \ref{main} for $\Psi_k$ with the Madsen-Weiss theorem of \cite{soren}.
\end{ex}


\section{Categorical model for $\Psi(M)\longrightarrow\Psi(\partial M)$}\label{model}\label{sec2}

In this section we prove the following.
\begin{prop}\label{catmodel} Let $M$ be a collared manifold of dimension $d$. There is a commutative diagram
\[\xymatrix{B\C\ar[d]_{res}&B\D\ar[d]_{res}\ar[l]_{\simeq}^\vartheta\ar[r]_-{\pr}& \Psi(M)\ar[d]_{res}\\
B\partial \C& B\partial \D\ar[l]^{\partial\vartheta}_{\simeq}\ar[r]_-{\pr}&\Psi(\partial M)
}\]
where the left vertical map is the realization of a functor of topological categories and the maps $\vartheta$ and $\partial \vartheta$ are weak equivalences.

If the orthogonal elements of $\Psi$ are almost open at $M$ (cf. \ref{almostopen} below) all the horizontal maps are weak equivalences.
\end{prop}

Define a topological category $\partial\C$ with objects space
\[\Ob\partial\C=\coprod_{a\in\mathbb{R}}\Psi(\partial M)^{\bot_{\mathbb{R}}}\]
The disjoint union is meant to indicate the fact that we are taking the discrete topology on the real coordinate.
For real numbers $0<\epsilon$ and $a\leq b$ define
\[\Mor_\epsilon(a,b)=\Psi(\partial M)^{\bot_{(-\infty,a+\epsilon)}}\cap \Psi(\partial M)^{\bot_{(b-\epsilon,\infty)}}\]
and let the morphism space of $\partial\C$ be
\[\Mor\partial\C=\coprod_{a\leq b}\colim\limits_{0<\epsilon}\Mor_\epsilon(a,b)\]
The source map of $\partial\C$ is induced by the maps
\[s_\epsilon \colon \Mor_\epsilon(a,b)\longrightarrow \Ob\partial\C\]
defined by sending $G$ to
\[s_\epsilon(G)=G|_{(-\infty,a+\epsilon)}\cup\alpha^{\ast}(G|_{(a-\epsilon,a+\epsilon)})\]
in the $a$-component,
for a choice of diffeomorphism $\alpha\colon (a-\epsilon,\infty)\longrightarrow (a-\epsilon,a+\epsilon)$ preserving the orientation.
By the orthogonality condition on the elements of $\Mor_\epsilon(a,b)$ the map $s_\epsilon$ is independent of the choice of $\alpha$, and it is continuous by continuity of the sheaf. Similarly one defines the target map as the colimit of maps
\[t_\epsilon(G)=\beta^{\ast}(G|_{(b-\epsilon,b+\epsilon)})\cup G|_{(b-\epsilon,\infty)}\]
for a choice of orientation preserving diffeomorphism $\beta\colon (-\infty,b+\epsilon)\longrightarrow(b-\epsilon,b+\epsilon)$.
Composition in $\partial\C$ is induced by the maps
\[\circ\colon\Mor_\epsilon(b,c)\times_{\Ob\partial\C}\Mor_{\epsilon}(a,b)\longrightarrow \Mor_{\epsilon}(a,c)\]
defined by
\[G\circ H=H|_{(-\infty,b-\epsilon)}\cup G|_{(b-\epsilon,\infty)}\]

\begin{rem} This is an "embedded cobordism category" for the sheaf $\Psi$, with objects embedded in $\partial M$ and cobordism direction given by the collar. For $\Psi=\Psi_k$ the sheaf of submanifolds of \cite{osmonoids}, and $M=\mathbb{R}^d\times (-\infty,0]$, the category $\partial\mathcal{C}_{\psi_k}(\mathbb{R}^{d}\times (-\infty,0])$ is almost the same as the cobordism category $\mathcal{C}_{k,d}$ of \cite{GMTW}. The difference is that our bordisms are not necessarily compact submanifolds of $\mathbb{R}^{d+1}$.
\end{rem}

Consider now the topological category $\C$ with objects
\[\Ob\C=\coprod_{a\in\mathbb{R}}\colim_{0<\epsilon}\Psi(M)^{(a-\epsilon,\infty)}\]
and morphisms
\[\Mor\C=\coprod_{a\leq b}\colim_{0<\epsilon}\Psi(M)^{\bot_{(a-\epsilon,a+\epsilon)}}\cap\Psi(M)^{\bot_{(b-\epsilon,\infty)}}\]
The source is induced by sending $G$ in the $a\leq b$ component to
\[s_\epsilon(G)=G|_{M_{<a+\epsilon}}\cup \alpha^{\ast}(G|_{(a-\epsilon,a+\epsilon)})\]
in the $a$-component, for an orientation preserving diffeomorphism $\alpha\colon (a-\epsilon,\infty)\longrightarrow (a-\epsilon,a+\epsilon)$. Here $M_{<u}\subset M_{<\infty}$ is the submanifold
\[M_{<u}= M_{<\infty}\backslash (\partial M\times [u,\infty))\]
The target projects off the $a$-coordinate, i.e. it sends $G$ in the $a\leq b$ component of the morphisms space to $G$ in the $b$-component of the objects space. Composition of $(b,c,G)$ and $(a,b,H)$ is defined to be just $(a,c,G)$ (notice that if $(b,c,G)$ and $(a,b,H)$ are composable then $H|_{M_{<b}}=G|_{M_{<b}}$).

The restriction map $\Psi(M)\longrightarrow \Psi(\partial M)$ induces a continuous functor \[\C\longrightarrow\partial\C\]

\begin{rem} Consider the functor $F\colon\partial\C\longrightarrow \Top$ that sends an object $(a,g)\in \Ob\partial\C$ to the set of extensions of $g$ to $M$
\[F(a,g):=\colim_{0<\epsilon} \{G\in \Psi(M)|G|_{(a-\epsilon,\infty)}=g|_{(a-\epsilon,\infty)}\}\]
and a morphism $(a\leq b,G)\in \Mor\partial\C$ to the map 
\[(-)|_{M_{<a+\epsilon}}\cup G|_{(a-\epsilon,\infty)}\colon F(a,s(G))\longrightarrow F(b,t(G))\]
The category $\C$ is the Grothendieck category $\partial\C\wr F$, and the restriction functor $\C\longrightarrow\partial\C$ corresponds to the projection functor $\partial\C\wr F\longrightarrow\partial\C$.
\end{rem}

The topological categories $\D$ and $\partial\D$ are associated to topological posets, defined as the spaces
\[\D=\coprod_{a\in\mathbb{R}}\colim_{0<\epsilon}\Psi(M)^{\bot_{(a-\epsilon,a+\epsilon)}}\]
and
\[\partial\D=\coprod_{a\in\mathbb{R}}\colim_{0<\epsilon}\Psi(\partial M)^{\bot_{(a-\epsilon,a+\epsilon)}}\]
with partial orders defined by $(a,G)\leq(b,H)$ if $a\leq b$ and $G=H$, on both spaces. The restriction map induces a map of topological posets
\[\res\colon \D\longrightarrow \partial\D\]

There is a functor $\partial\vartheta\colon \partial\D\longrightarrow \partial\C$ that sends a morphism $(a\leq b,G)$ with $G\in\Psi(\partial M)^{\bot_{(a-\epsilon,a+\epsilon)}}\cap\Psi(\partial M)^{\bot_{(b-\epsilon,b+\epsilon)}}$ to
\[\partial\vartheta(a\leq b,G)=(a\leq b,\alpha^{\ast}(G|_{(a-\epsilon,a+\epsilon)})\cup G|_{(a-\epsilon,b+\epsilon)}\cup\beta^{\ast}(G|_{(b-\epsilon,b+\epsilon)}))\]
for diffeomorphisms $\alpha\colon (-\infty,a+\epsilon)\longrightarrow(a-\epsilon,a+\epsilon)$ and $\beta\colon(b-\epsilon,\infty)\longrightarrow (b-\epsilon,b+\epsilon)$ preserving the orientation.
Similarly, there is a functor \[\vartheta\colon \D\longrightarrow \C\] that sends a morphism $(a\leq b,G)$ with $G\in \Psi( M)^{\bot_{(a-\epsilon,a+\epsilon)}}\cap\Psi( M)^{\bot_{(b-\epsilon,b+\epsilon)}}$ to 
\[\vartheta(a\leq b,G)=(a\leq b,G|_{(-\infty,b+\epsilon)}\cup\beta^{\ast}(G|_{(b-\epsilon,b+\epsilon)}))\]

\begin{prop}\label{cd} The square
\[\xymatrix{B\C\ar[d]_{\res}&B\D\ar[l]^{\vartheta}\ar[d]^{\res}\\
B\partial\C&B\partial\D\ar[l]^{\partial\vartheta}
}\]
commutes and the maps $\vartheta$ and $\partial \vartheta$ are weak homotopy equivalences.
\end{prop}

\begin{proof} The square clearly commutes. We show that $\vartheta$ is a weak equivalence, the same result for $\partial \vartheta$ is completely analogous. We show that the nerve of $\vartheta$ is a levelwise homotopy equivalence of simplicial spaces
\[\mathcal{N}_{\sbt}\vartheta\colon \mathcal{N}_{\sbt}\D\longrightarrow\mathcal{N}_{\sbt}\C\]
An element of $\mathcal{N}_k\C$ is represented by a sequence of real numbers $a_0\leq\dots\leq a_k$, a $0<\epsilon$ and a $G\in\Psi(M)$ with $G\bot_{(a_i-\epsilon,a_i+\epsilon)}$ all $0\leq i\leq k-1$ and $G\bot_{(a_k-\epsilon,\infty)}$. This clearly defines an element of $\mathcal{N}_{k}\D$, and actually a natural (non-simplicial) inclusion
\[\iota_k\colon\mathcal{N}_k\C\longrightarrow\mathcal{N}_{k}\D\]
which is a section for $\mathcal{N}_{k}\vartheta$.
The composite $\iota_k\circ\mathcal{N}_{k}\vartheta$ sends an element $(a_0\leq \dots\leq a_k,G)$ for $G\in \Psi(M)$ with $G\bot_{(a_i-\epsilon,a_i+\epsilon)}$ to
\[(a_0\leq\dots\leq a_k,G|_{(-\infty,a_{k}+\epsilon)}\cup \alpha^{\ast}(G|_{(a_k-\epsilon,a_k+\epsilon)}))\]
for a diffeomorphism $\alpha\colon (a_k-\epsilon,\infty)\longrightarrow(a_k-\epsilon,a_k+\epsilon)$ that preserves the orientation.
Take a continuous family of embeddings $\alpha_{t}\colon(a_k-\epsilon,\infty)\longrightarrow (a_k-\epsilon,\infty)$, for $t\in [0,1]$, satisfying
\begin{enumerate}
\item $\im(\alpha_{t})=(0,a_{k}+\epsilon+\frac{t}{1-t})$
\item $\alpha_{0}=\alpha$
\item $\alpha_{1}=\id$
\end{enumerate}
The embedding $\alpha_t$ induces a continuous map $\mathcal{N}_{k}\D\longrightarrow \mathcal{N}_{k}\D$
by sending $(a_0\leq\dots\leq a_k,G)$ to
\[(a_0\leq\dots\leq a_k,G|_{(-\infty,a_{k}+\epsilon)}\cup \alpha_{t}^{\ast}(G|_{(a_k-\epsilon,a_k+\epsilon)}))\]
These maps assemble into a homotopy
\[\mathcal{N}_{k}\D\times [0,1]\longrightarrow\mathcal{N}_{k}\D\]
from $\iota_k\circ\mathcal{N}_{k}\vartheta$ to the identity.
\end{proof}

For a real number $a\in\mathbb{R}$ we denote 
\[\Psi(\partial M)^{\bot a}:=\colim_{0<\epsilon}\Psi(\partial M)^{\bot_{(a-\epsilon,a+\epsilon)}}\]
and define $\Psi(M)^{\bot a}$ in a similar way.

\begin{defn}\label{almostopen}
We say that \textbf{orthogonal elements of $\Psi$ are almost open at $M$} if there is an open cover $\{U_a\}_{a\in\mathbb{R}}$ of $\Psi(M)$ such that
\begin{enumerate}
\item $\Psi(\partial M)^{\bot_{(a-\epsilon,a+\epsilon)}}\subset U_a$ for every $\epsilon>0$,
\item The image of $U_a$ by the restriction map $U_a|_{\partial M}\subset \Psi(\partial M)$ is open,
\item For every finite sequence $a_0\leq\dots\leq a_k$ the maps induced by the inclusions
\[\cap_{i=0}^k\Psi(M)^{\bot a_i}\longrightarrow \cap_{i=0}^kU_{a_i}\]
and 
\[\cap_{i=0}^k\Psi(\partial M)^{\bot a_i}\longrightarrow \cap_{i=0}^kU_{a_i}|_{\partial M}\]
are weak equivalences.
\end{enumerate}
\end{defn}

\begin{ex} In the case of the sheaf $\Psi_k$, one can essentially choose $U_a$ to be the subspace of $\Psi(M)$ of submanifolds that intersect $\partial M\times a$ transversally. These sets $U_a$ are open only when $M$ is compact. For non-compact $M$ one needs to modify this sets slightly (cf. \ref{transalmostopen}).
\end{ex}

\begin{prop}\label{segaltrick} Suppose that orthogonal elements of $\Psi$ are almost open at $M$.
Then the projections from $\D$ to $\Psi(M)$ and from $\partial\D$ to $\Psi(\partial M)$ define a commutative square
\[\xymatrix{B\D\ar[r]^{\simeq}\ar[d]_{res}&\Psi(M)\ar[d]^{res}\\
B\partial\D\ar[r]_{\simeq}&\Psi(\partial M)
}\]
where the horizontal maps are weak equivalences.
\end{prop}

\begin{proof}

The projections $\D\longrightarrow\Psi(M)$ and $\partial\D\longrightarrow\Psi(\partial M)$ factor trough the topological posets $\mathcal{E}(M)=\coprod_{a\in\mathbb{R}}U_a$ and $\partial\mathcal{E}(M)=\coprod_{a\in\mathbb{R}}U_a|_{\partial M}$ with partial orders again defined as $(a,G)\leq(b,H)$ if $a\leq b$ and $G=H$. The inclusions induce weak equivalences $B\D\simeq B\mathcal{E}(M)$ and $B\partial\D\simeq B\partial\mathcal{E}(M)$. Therefore it is enough to show that the inclusions $U_a\subset \Psi(M)$ induce weak equivalences 
\[B\mathcal{E}(M) \stackrel{\simeq}{\longrightarrow}\Psi(M) \ \ \ \mbox{ and } \ \ \ B\partial\mathcal{E}(M)\stackrel{\simeq}{\longrightarrow}\Psi(\partial M)\]

The nerves of $\Psi(M)$ and $\Psi(\partial M)$ are constant simplicial spaces. By \cite{soren}*{§3.4}, which is a corollary of \cite{segal}*{§A1}, it is enough to show that the projections $\mathcal{N}_{\sbt}\mathcal{E}(M)\longrightarrow \Psi(M)$ and $\mathcal{N}_{\sbt}\partial\mathcal{E}(M)\longrightarrow\Psi(\partial M)$ are levelwise étale (that is, are open maps and local homeomorphisms) and that their fibers have contractible realizations. We show the result only for $\mathcal{E}(M)\longrightarrow\Psi(M)$, the other case being completely analogous.

The fiber $F_{\sbt}$ over a fixed element $G\in\Psi(M)$ is the simplicial space with $k$-simplicies the space
\[F_k=\{(a_0\leq\dots\leq a_k)|G\in U_{a_i}\mbox{ for all }0\leq i\leq k\}\]
Notice that since the topology on the real coordinates is discrete, this a discrete space.
Since the sets $U_a$ cover $\Psi(M)$, the spaces $F_k$ are non-empty. Moreover $F_{\sbt}$ is the nerve of the totally ordered set $\{a\in\mathbb{R}| G\in U_a\}$, and therefore its realization is contractible.

It remains to show that the projections at levelwise étale. Let us show that it is an open map.
Since the real coordinates are discrete, an open subset of $\mathcal{N}_k\mathcal{E}(M)$ is a union of sets of the form $\{(a_0\leq\dots\leq a_k)\}\times V$, with $V$ open in $\cap_{i=0}^kU_{a_i}$. Its image by the projection map is the union of the sets $V$. Since $V$ is open in $\cap_{i=0}^kU_{a_i}$ and $\cap_{i=0}^kU_{a_i}$ is open in $\Psi(M)$, the image $V$ is open in $\Psi(M)$. It remains to show that the projection is locally injective. Since the real coordinate is discrete, the set $\{(a_0\leq\dots\leq a_k)\}\times \cap_{i=0}^kU_{a_i}$ is an open neighborhood of a general element $((a_0\leq\dots\leq a_k),G)\in \mathcal{N}_k\mathcal{E}(M)$. The restriction of the projection to $\{(a_0\leq\dots\leq a_k)\}\times \cap_{i=0}^kU_{a_i}$ is injective.
\end{proof}


\section{The homotopy fiber of $\Psi(M)\longrightarrow \Psi(\partial M)$}\label{sec3}

In this section we use the categorical model for the restriction map $\Psi(M)\longrightarrow\Psi(\partial M)$ described in the previous section to show the following.
\begin{prop}\label{identifhtpyfib}
Let $M$ be a collared $d$-dimensional manifold. Suppose that the elements of $\Psi(M)$ are almost open, and $\Psi$ is group-like and damping at $M$ (cf. \ref{gplike} and \ref{damping} below). Then the homotopy fiber of $\Psi(M)\longrightarrow\Psi(\partial M)$ over an element $g\in\Psi(\partial M)^{\bot_{\mathbb{R}}}$ is naturally weakly equivalent to
\[\Psi(M;g):=\colim_{0<\epsilon}\{G\in\Psi(M): \ G|_{(-\epsilon,\infty)}=g|_{(-\epsilon,\infty)}\}\]
\end{prop}

The conditions on the sheaf $\Psi$ above allows us to identify the homotopy fiber of the restriction functor $\C\longrightarrow \partial \C$ over an object $(g,0)$ as the space $\Psi(M;g)$. This is carried out by means of Quillen theorem B for simplicial categories, proved by Waldhausen in \cite{wald}. We recall its content.

Let $F\colon C\longrightarrow D$ be a simplicial functor. Given an object $Y_q\in\Ob D_q$, define the (right) fiber of $F$ over $Y_q$ by the simplicial category $Y_q/F$, whose level $n$ category has objects
\[\Ob(Y_q/F)_n=\coprod_{\nu\colon [n]\longrightarrow[q]}\{(X\in\Ob C_n, b\in\Mor D_n)|b\colon F(X)\longrightarrow \nu^\ast Y_q\}\]
The morphisms from $(X,b)$ to $(X',b')$ are morphisms $a\colon X\longrightarrow X'$ of $C_n$ making the diagram
\[\xymatrix{F(X)\ar[d]_{F(a)}\ar[r]^b &\nu^\ast Y_q\\
F(X')\ar[ur]_{b'}}\]
commutative.

There are two kind of maps between the fibers over the objects of $D$. A map of the first kind is a map $\beta_\ast\colon Y_q/F\longrightarrow Y'_q/F$ induced by a morphism $\beta\colon Y_q\longrightarrow Y'_q$ of $D_q$. It maps a morphism $a\colon (X,b)\longrightarrow (X',b')$ of $(Y_q/F)_n$ to the morphism $a\colon (X,\nu^\ast(\beta)\circ b)\longrightarrow (X',\nu^\ast(\beta)\circ b')$ of $(Y'_q/F)_n$
\[\xymatrix{F(X)\ar[d]_{F(a)}\ar[r]^b &\nu^\ast Y_q\ar[r]^{\nu^\ast\beta}&\nu^\ast Y'_q\\
F(X')\ar[ur]_{b'}}\]

A map of the second kind is a map $\alpha_\ast\colon \alpha^\ast( Y_q)/F\longrightarrow Y_q/F$ induced by a morphism $\alpha\colon [p]\longrightarrow [q]$ in $\Delta$. It maps a morphism $a\colon (X,b)\longrightarrow (X',b')$ in the $\nu\colon[n]\longrightarrow[p]$ component of $(\alpha^\ast( Y_q)/F)_n$ to $a\colon (X,b)\longrightarrow (X',b')$ in the $\alpha\circ\nu$ component of $( Y_q/F)_n$.
\begin{prop}[\cite{wald}]\label{quillenB}
Let $F\colon C\longrightarrow D$ be a simplicial functor for which all the maps of the first and the second kind induce weak equivalences on the classifying space. The for every object $Y\in D_0$ the diagram
\[\xymatrix{B(Y/F)\ar[d]\ar[r]&BC\ar[d]^{BF}\\
\ast\ar[r]_{Y}&BD
}\]
is homotopy cartesian.
\end{prop}
We briefly recall that the classifying space of a simplicial category is defined as the realization of the bisimplicial set obtained by taking the nerve levelwise. This realization is defined as the realization of the simplicial space obtained by levelwise realizing one of the simplicial directions. This is the same as realizing the diagonal simplicial set.

We want to apply this Quillen B theorem to describe the homotopy fibers of the restriction functor $B\C\longrightarrow B\partial \C$. The singular functor $\sing_{\sbt}$ from topological spaces to simplicial sets induces a functor from topological categories to simplicial categories. Given a topological category $C$, define a simplicial category $\sing_{\sbt}C$ in simplicial degree $k$ by $\Ob\sing_kC:=\sing_k\Ob C$ and morphisms $\Mor\sing_kC:=\sing_k\Mor C$. Source, target and composition maps are defined pointwise. In our situation, this construction induces a simplicial functor
\[\res\colon\sing_{\sbt}\C\longrightarrow \sing_{\sbt}\partial \C\]
Before showing that under appropriate conditions this functor satisfies theorem B, let us identify the classifying space of the fiber over an object $(b,g)\in\sing_0\partial\C=\partial C$. Let $F_{(b,g)}$ be the topological poset with underlying space
\[F_{(b,g)}=\coprod_{a\leq b}\colim_{0<\epsilon}\{G\in\Psi(M)^{\bot_a}|G|_{(b-\epsilon,\infty)}=g|_{(b-\epsilon,\infty)}\}\] 
and partial order defined by $(a,G)\leq (a',H)$ if $a\leq a'$ and $G=H$. 
 Notice that by definition $(b,g)/\res=\sing_{\sbt}F_{(b,g)}$.

\begin{prop}\label{htpyfiber} If orthogonal elements of $\Psi(M)$ are almost open, for every $g\in\Psi(\partial M)^{\bot_{\mathbb{R}}}$ there is a natural zig-zag of weak equivalences
\[\xymatrix{B((0,g)/\res)\ar[r]_-{\simeq}& BF_{(0,g)}\ar[r]_-{\simeq}^-{\pr}&\Psi(\partial M;g)}\]
where the right-hand map is defined by projecting off the real coordinate.
\end{prop}

\begin{proof}
As we noticed, $(b,g)/\res=\sing_{\sbt}F_{(b,g)}$ and therefore the evaluation map induces a topological functor
\[|(b,g)/\res|\longrightarrow F_{(b,g)}\]
that is a levelwise equivalence on the nerve. The projection map $\pr\colon F_{(0,g)}\longrightarrow\Psi(\partial M;g)$ is an equivalence by the same argument of \ref{segaltrick}.
\end{proof}

We now see under which assumptions on $\Psi$ the restriction functor satisfies the conditions of Quillen's theorem \ref{quillenB}.

For $a\in\mathbb{R}$ denote 
\[\Psi(\partial M)^{\bot_a}=\colim_{0<\epsilon}\Psi(\partial M)^{\bot_{(a-\epsilon,a+\epsilon)}}\]
Given a sequence $a_0\leq\dots\leq a_k$ denote $\Psi(\partial M)^{\bot_{a_0,\dots, a_k}}:=\cap_{i=0}^k\Psi(\partial M)^{\bot_{a_i}}$, and define $\Psi(M)^{\bot_a}$ and $\Psi(M)^{\bot_{a_0,\dots, a_k}}$ analogously. 
For every real number $u$ denote $t_u\colon \mathbb{R}\longrightarrow \mathbb{R}$ the diffeomorphism $t_u(x)=x+u$

\begin{defn}\label{damping}
The sheaf $\Psi$ is \textbf{damping at $M$} if given any $f\colon\Delta^q\longrightarrow \Psi(\partial M)^{\bot_{\mathbb{R}}}$ there is a family $\{B^{\phi,\psi}\colon\Delta^{n}\longrightarrow\Psi(\partial M)|\phi,\psi\colon[n]\longrightarrow [q]\}$ such that
\begin{enumerate}[i)]
\item $B^{\phi,\psi}(\sigma)|_{(-\infty,\epsilon)}=\phi^{\ast}f(\sigma)|_{(-\infty,\epsilon)}$
\item $t_{1}^{\ast}(B^{\phi,\psi}(\sigma)|_{(1-\epsilon,\infty)})=\psi^{\ast}f(\sigma)|_{(-\epsilon,\infty)}$
\item $B^{\phi,\phi}=\phi^{\ast}f$
\item $\omega^{\ast}B^{\phi,\psi}=B^{\omega^{\ast}\phi,\omega^{\ast}\psi}$ for all $\omega\colon [m]\longrightarrow [n]$
\end{enumerate}
\end{defn}

\begin{rem}
Here's an illustration of this strange damping condition. Given $f\colon\Delta^q\longrightarrow \Psi(\partial M)^{\bot_{\mathbb{R}}}$ and $\phi,\psi\colon [n]\longrightarrow[q]$, consider the map $H^{\phi,\psi}\colon\Delta^n\times[0,1]\longrightarrow \Psi(\partial M)^{\bot_{\mathbb{R}}}$ defined by
\[H^{\phi,\psi}(x,t):=f(\phi_\ast(x)\cdot t+\psi_{\ast}(x)\cdot(1-t))\]
This is a homotopy from $\phi^{\ast}f$ to $\psi^{\ast}f$, and this family satisfies \[\omega^{\ast}H_{t}^{\phi,\psi}=H_{t}^{\omega^{\ast}\phi,\omega^{\ast}\psi}\]
for all $\omega\colon [m]\longrightarrow [n]$ and fixed $t\in[0,1]$. A family $B^{\phi,\psi}\colon\Delta^{n}\longrightarrow\Psi(\partial M)$ for the damping condition exists if each element $H^{\phi,\psi}(x,t)\in \Psi(\partial M)^{\bot_{\mathbb{R}}}$ is determined by the restriction to a neighborhood of $\partial M\times t$ of $B^{\phi,\psi}(x)$. We show that the sheaf of submanifolds $\Psi_k$ satisfies this property in \ref{psikdamping}, by means of a smooth approximation theorem.
\end{rem}

\begin{prop}\label{secondkindzero}
Suppose that $\Psi$ is damping at $M$. Then for every object $Y_q$ of $\sing_q\partial\C$ the map the second kind
\[(\eta_0)_\ast\colon \eta_{0}^{\ast}Y_q/\res\longrightarrow Y_q/\res\]
induced by the value zero map $\eta_0\colon [0]\longrightarrow [q]$ is a weak equivalence on classifying spaces.
\end{prop}

\begin{proof} 

We show that for a fixed $k$ the map of simplicial sets
\[\mathcal{N}_k(\eta_0)_\ast\colon \mathcal{N}_k(\eta_{0}^{\ast}Y_q/\res)_{\sbt}\longrightarrow \mathcal{N}_k(Y_q/\res)_{\sbt}\]
is a simplicial homotopy equivalence. Let us assume for simplicity that $Y_q=(0,f)$ for some $f\colon\Delta^k\longrightarrow\Psi(\partial M)^{\bot_{\mathbb{R}}}$. The argument for a general $a\neq 0$ is completely analogous.
An element of $\mathcal{N}_k(Y_q/\res)_{n}$ is a family $(\phi,\underline{a};G)$ of a sequence of real numbers $\underline{a}=(a_0\leq\dots\leq a_{k}\leq 0)$, a map $\phi\colon [n]\longrightarrow [q]$ and a continuous $G\colon \Delta^{n}\longrightarrow\Psi(\partial M)^{\bot_{a_0,\dots,a_{k}}}$ such that
\[G(\sigma)|_{(-\epsilon,\infty)}=(\phi^{\ast}f)(\sigma)|_{(-\epsilon,\infty)}\]
for all $\sigma\in\Delta^n$ and some $\epsilon>0$ depending on $\sigma$.

Pick a family $\{B^{\phi,\psi}|\phi,\psi\colon[n]\longrightarrow [q]\}$ for $f$ given by the damping condition \ref{damping}.
Define a simplicial map
\[\xi\colon \mathcal{N}_k(Y_q/\res)_{\sbt}\longrightarrow \mathcal{N}_k(\eta_{0}^{\ast}Y_q/\res)_{\sbt}\]
by sending an element $(\phi,\underline{a};G)$ as above to
\[\xi(\phi,\underline{a};G)=(p\colon [n]\longrightarrow [0],\underline{a};\xi(G))\]
where $p\colon [n]\longrightarrow [0]$ is the unique map, and $\xi(G)\colon\Delta^n\longrightarrow \Psi(\partial M)^{\bot_{a_0,\dots,a_k}}$ is defined by
\[\xi(G)(\sigma):=s_{\underline{a}}^{\ast}(G(\sigma)|_{(-\infty,\epsilon)}\cup B^{\phi,\eta_0\circ p}(\sigma)|_{(-\epsilon,1+\epsilon)})\]
Here $s_{\underline{a}}\colon \mathbb{R}\longrightarrow \mathbb{R}$ is a diffeomorphism that sends $0$ to $1$, and with the following properties. If  $\underline{a}=(a_0=\dots=a_k=0)$ simply take $s_{\underline{a}}$ to be $t_1=(-)+1$. If there is a coordinate $a_i< 0$, let $a_m$ be the biggest among them, and choose $s_{\underline{a}}$ so that it agrees with the identity in a neighborhood of $(-\infty,a_m]$, and with $t_1$ on a neighborhood of $[0,\infty)$.
Condition i) and ii) of the damping property insure that the gluing is well defined, and that $\xi(G)$ agrees with $p^\ast (\eta_{0}^\ast f)$ around $0$ as required for belonging in $\mathcal{N}_k(\eta_{0}^{\ast}Y_q/\res)_n$. 
Our choice of $s_{\underline{a}}$ guarantees that $\xi(G)(\sigma)$ is still orthogonal at $a_0,\dots a_k$. By condition iv) on $\{B^{\phi,\psi}\}$ the map $\xi$ is simplicial.

The Composite $\eta_0\circ\xi$ sends $(\phi,\underline{a};G)$ to $(\eta_0\circ p,\underline{a};\xi(G))$. Define a simplicial homotopy  
\[H\colon \mathcal{N}_kY_q/\res_{\sbt}\times\Delta[1]\longrightarrow \mathcal{N}_kY_q/\res_{\sbt}\]
by sending $(\phi,\underline{a};G)$ and $b\colon [n]\longrightarrow [1]$ to 
\[H((\phi,\underline{a};G),b)=(\phi\cdot b,\underline{a};H_b(G))\]
where $\phi\cdot b$ is pointwise multiplication of maps $[n]\longrightarrow[1]$, and
\[H_b(G)(\sigma)=s_{\underline{a}}^{\ast}(G(\sigma)|_{(-\infty,\epsilon)}\cup B^{\phi,\phi\cdot b}(\sigma)|_{(-\epsilon,1+\epsilon)})\]
This is a simplicial homotopy from $\eta_0\circ\xi$ to a map $H_1$ that sends $(\phi,\underline{a};G)$ to $(\phi,\underline{a};H_1(G))$ with
\[H_1(G)(\sigma)=s_{\underline{a}}^{\ast}(G(\sigma)|_{(-\infty,\epsilon)}\cup (\phi^{\ast}f)(\sigma)|_{(-\epsilon,1+\epsilon)})\]
Here we used that $B^{\phi,\phi}=\phi^{\ast}f$. Thus $H_1(G)$ is a scaling $G$. We define a homotopy that inflates this back to $G$. For every $\underline{a}$, take a family of diffeomorphisms $s_{\underline{a}}^{u}\colon \mathbb{R}\longrightarrow\mathbb{R}$ continuous in $u\in[0,1]$ such that $s_{\underline{a}}^{0}=s_{\underline{a}}$, $s_{\underline{a}}^{1}=\id$ and such that each $s_{\underline{a}}^{u}$ agrees with the identity on a neighborhood of $(-\infty,a_m]$, and with $t_u=(-)+u$ on a neighborhood of $[0,\infty)$.
Define $K\colon \mathcal{N}_kY_q/\res_{\sbt}\times\Delta[1]\longrightarrow \mathcal{N}_kY_q/\res_{\sbt}$ by sending 
$((\phi,\underline{a};G),b\colon[n]\longrightarrow [1])$ to $(\phi,\underline{a};K_b(G))$ with
\[K_b(G)(\sigma)=(s_{\underline{a}}^{b_\ast(\sigma)})^{\ast}(G(\sigma)|_{(-\infty,\epsilon)}\cup (\phi^{\ast}f)(\sigma)|_{(-\epsilon,1+\epsilon)})\]
where $b_\ast\colon\Delta^n\longrightarrow \Delta^1\cong[0,1]$ is the induced map.
This is a simplicial homotopy from $H_1$ to the identity.

The other composite $\xi\circ(\eta_0)_\ast$ sends $(p,\underline{a};G)$ to $(p,\underline{a};\xi(G))$ with
\[\xi(G)(\sigma)=s_{\underline{a}}^{\ast}(G(\sigma)|_{(-\infty,\epsilon)}\cup p^\ast(\eta_{0}^{\ast}f)(\sigma)|_{(-\epsilon,1+\epsilon)})\]
There is a simplicial homotopy similar to $K$ from this map to the identity.
\end{proof}

\begin{defn}\label{inverse}\label{gplike}
A \textbf{right inverse} for an element $G\in\Psi(\partial M)^{\bot_{0,1}}$ is an element $R\in\Psi(\partial M)$ such that
\begin{enumerate}[i)]
\item $R|_{(1-\epsilon,1+\epsilon)}=G|_{(1-\epsilon,1+\epsilon)}$
\item $t_{2}^{\ast}(R)|_{(-\epsilon,\epsilon)}=G|_{(-\epsilon,+\epsilon)}$ 
\item There is a path $\gamma\colon[0,1]\longrightarrow\Psi(\partial M)$ with
\begin{itemize}
\item $\gamma(0)=G|_{(-\infty,1-\epsilon)}\cup R|_{(1-\epsilon,\infty)}$
\item $\gamma(1)\in\Psi(\partial M)^{\bot_{(-\epsilon,2+\epsilon)}}$
\item $\gamma(t)|_{(-\infty,+\epsilon)}=G|_{(-\infty,+\epsilon)}$ and $\gamma(t)|_{(2-\epsilon,\infty)}=R|_{(2-\epsilon,\infty)}$
\end{itemize} 
\end{enumerate}
A \textbf{left inverse} for $G\in\Psi(\partial M)^{\bot_{0,1}}$ is a $L\in\Psi(\partial M)$ satisfying
\begin{enumerate}[i)]
\item $L|_{(-\epsilon,\epsilon)}=G|_{(-\epsilon,+\epsilon)}$
\item $t_{-2}^{\ast}(L)|_{(1-\epsilon,1+\epsilon)}=G|_{(1-\epsilon,1+\epsilon)}$ 
\item There is a path $\delta\colon[0,1]\longrightarrow\Psi(\partial M)$ with
\begin{itemize}
\item $\delta(0)=L|_{(-\infty,\epsilon)}\cup G|_{(-\epsilon,\infty)}$
\item $\delta(1)\in\Psi(\partial M)^{\bot_{(-1-\epsilon,1+\epsilon)}}$
\item $\delta(t)|_{(-\infty,-1+\epsilon)}=L|_{(-\infty,-1+\epsilon)}$ and $\gamma(t)|_{(1-\epsilon,\infty)}=G|_{(1-\epsilon,\infty)}$
\end{itemize}
\end{enumerate}

A sheaf $\Psi\colon\emb_{d}^{op}\longrightarrow \Top$ is called \textbf{group like at $M$} if every element in $G\in\Psi(\partial M)^{\bot_{(-\infty,\epsilon)}}\cap \Psi(\partial M)^{\bot_{(1-\epsilon,\infty)}}$ has both a right and a left inverse.
\end{defn}

Notice that the gluing conditions force a right inverse $R$ to belong to $\Psi(\partial M)^{\bot_{1,2}}$, and similarly $L\in \Psi(\partial M)^{\bot_{-1,0}}$. Also, if $R$ is a right inverse for $G$ then $G$ is a left inverse for $R$.

\begin{ex}
We show in \ref{mangplike} that the sheaf $\Psi_k$ is group like at $M$ if $\partial M=N\times\mathbb{R}$ for some manifold without boundary $N$ of dimension $d-2$. The idea is to define a left inverse of a submanifold $G\in\Psi_k(\partial M)^{\bot_{0,1}}$ as the flipped manifold $(-\id)^\ast G$. Similarly $(2-\id)^\ast G$ is going to be a right inverse for $G$.
\end{ex}

\begin{prop}\label{firstkind}
Suppose that $\Psi$ is group like and damping at $M$, and that orthogonal elements of $\Psi(M)$ are almost open. Then all the maps of the first kind for $\res\colon\sing_{\sbt}\C\longrightarrow \sing_{\sbt}\partial \C$ are weak equivalences.
\end{prop}

\begin{proof}
Let $\beta$ be a morphism in $\partial\C_q$ from $Y_q$ to $Y_{q}'$. The value zero map $\eta_0\colon [0]\longrightarrow [q]$ induces a commutative diagram of simplicial functors
\[\xymatrix{Y_q/\res\ar[r]^{\beta_\ast}&Y_{q}'/\res\\
\eta_{0}^{\ast}Y_q\ar[u]^{(\eta_0)_\ast}_{\simeq}\ar[r]_{(\eta_{0}^{\ast}\beta)_\ast}&\eta_{0}^{\ast}Y_{q}'\ar[u]^{(\eta_0)_\ast}_{\simeq}
}\]
where $\eta_{0}^{\ast}\beta$ is a morphism in $\partial\C_0$, and the vertical maps are equivalences proposition \ref{secondkindzero} above. It is therefore enough to show that the maps of the first kind induced by morphisms in degree zero are equivalences.

Let $(b,g\in\Psi(\partial M)^{\bot_{\mathbb{R}}})$ be an object of $\sing_0\partial \C=\partial\C$. Recall the topological poset
\[F_{(b,g)}=\coprod_{a\leq b}\colim_{0<\epsilon}\{H\in\Psi(M)^{\bot_{a}}|G|_{(b-\epsilon,\infty)}=g|_{(b-\epsilon,\infty)}\}\]
A morphism $\beta=(b\leq c,G)$ of $\partial\C$ from $(b,g)$ to $(c,h)$ induces a commutative diagram
\[\xymatrix{|(b,g)/\res|\ar[r]^{|\beta_{\ast}|}\ar[d]^{\simeq}&|(c,h)/\res|\ar[d]^{\simeq}\\
F_{(b,g)}\ar[r]_{G_\ast}&F_{(c,h)}
}\]
where $G_\ast$ sends a morphism $(a\leq a',H)$ of $F_{(b,g)}$ to $(a\leq a',H|_{(-\infty,b+\epsilon)}\cup G|_{(b-\epsilon,\infty)})$. The vertical maps are induced by evaluation, and are levelwise weak equivalences on the nerve as showed in \ref{htpyfiber}.
Now let $F_{(b,g)}^{t}$ be the topological poset with same underlying set as $F_{(b,g)}$, but topologized using the standard topology of $\mathbb{R}$ on the real coordinate $a\leq b$. The inclusion of $\mathbb{R}$ with the discrete topology into $\mathbb{R}$ with the standard topology induces a commutative diagram
\[\xymatrix{
F_{(b,g)}\ar[d]_{\simeq}\ar[r]_{G_\ast}&F_{(c,h)}\ar[d]^{\simeq}\\
F^{t}_{(b,g)}\ar[r]_{G_\ast}&F^{t}_{(c,h)}
}\]
where the vertical maps are equivalences on the realization by lemma \ref{topdiscronfib} below. Therefore it is enough to show that on the nerve
\[G_\ast\colon \mathcal{N}_kF^{t}_{(b,g)}\longrightarrow\mathcal{N}_kF^{t}_{(c,h)}\]
is a homotopy equivalence for all $k$. We assume for simplicity that $b=0$ and $c=1$, the general argument being a translation of this case.

Take a right inverse $R$ for $G$ in the sense of definition \ref{inverse}. Define a map
\[G^{\ast}\colon \mathcal{N}_kF^{t}_{(1,h)}\longrightarrow\mathcal{N}_kF^{t}_{(0,g)}\]
by sending $(a_0\leq\dots\leq a_k\leq 1;H)$ to $(\underline{a}-2=(a_0-2\leq\dots\leq a_k-2\leq 0);G^{\ast}(H))$ where
\[G^{\ast}(H)=t_{2}^{\ast}(H|_{(-\infty,1+\epsilon)}\cup R|_{(1-\epsilon,2+\epsilon)}\cup g|_{(2-\epsilon,\infty)})\]
The map $G^{\ast}\circ G_\ast$ sends a $(\underline{a} \leq 0;H)$ to $(\underline{a}-2\leq 0,G^{\ast}G_\ast H)$ with
\[G^{\ast}G_\ast H=t_{2}^{\ast}(H|_{(-\infty,+\epsilon)}\cup G|_{(-\epsilon,1+\epsilon)}\cup R|_{(1-\epsilon,2+\epsilon)}\cup g|_{(2-\epsilon,\infty)})\]
Define a map $K\colon \mathcal{N}_kF^{t}_{(0,g)}\times I\longrightarrow \mathcal{N}_kF^{t}_{(0,g)}$ by sending $(\underline{a};H)$ and $u\in[0,1]$ to $(\underline{a}-2;K_u(H))$ with 
\[K_u(H)=t_{2}^{\ast}(H|_{(-\infty,+\epsilon)}\cup \gamma(u)|_{(-\epsilon,\infty)})\]
where $\gamma$ is a path for the right inverse (cf. \ref{inverse}).
This defines a homotopy from $G^{\ast}\circ G_\ast$ to a map $K_1$ given by
\[K_1(H)=t_{2}^{\ast}(H|_{(-\infty,+\epsilon)}\cup g|_{(-\epsilon,\infty)})\]
This is because $\gamma(1)$ is orthogonal on $(-\epsilon,2+\epsilon)$, and it agrees with $G$ on $(-\epsilon,\epsilon)$ which is equal to $g|_{(-\epsilon,\epsilon)}$. Now $g$ is orthogonal on all $\mathbb{R}$, and therefore $\gamma(1)$ must agree with $g$ on $(-\epsilon,2+\epsilon)$.
 Finally define a homotopy $K'$ that sends $(\underline{a};H)$ and $u\in[0,1]$ to $(\underline{a}-2+2u,K_{u}'(H))$ with 
\[K_{u}'(H)=t_{2-2u}^{\ast}(H|_{(-\infty,+\epsilon)}\cup g|_{(-\epsilon,\infty)})\]
This shows that $G^{\ast}$ is a left homotopy inverse for $G_\ast$. Similarly a left inverse $L$ for $G$ defines a right homotopy inverse for $G_\ast$.
\end{proof}

\begin{prop}\label{secondkind}
Suppose that $\Psi$ is group like and damping at $M$, and that orthogonal elements of $\Psi(M)$ are almost open. Then all the maps of the second kind for $\res\colon\sing_{\sbt}\C\longrightarrow \sing_{\sbt}\partial \C$ are weak equivalences.
\end{prop}

\begin{proof}
Notice that for a general map $\alpha\colon [p]\longrightarrow[q]$ there is a commutative diagram
\[\xymatrix{\alpha^{\ast}Y_q/\res\ar[r]^{\alpha_\ast}&Y_q/\res\\
\eta^{\ast}(\alpha^{\ast}Y_q)/\res\ar[u]^{\eta_\ast}\ar[ur]_{(\alpha\circ\eta)_\ast}&
}\]
for any choice of $\nu\colon[0]\longrightarrow [p]$, and $\alpha\circ\nu$ has also source $[0]$. Therefore it is enough to show that every map with source zero $\eta\colon [0]\longrightarrow [q]$ induces a weak equivalence $\eta_\ast$. Suppose that $Y_q=(0,f)$ for a map $f\colon \Delta^{n}\longrightarrow \Psi(\partial M)^{\bot_{\mathbb{R}}}$, and consider the diagram
\[\xymatrix{ \eta^{\ast}(0,f)/\res\ar[r]^{\eta_\ast }\ar[d]_{B^{\eta\circ p,\eta_0\circ p}_\ast}&(0,f)/\res\\
\eta_{0}^{\ast}(0,f)/\res\ar[ur]_{(\eta_0)_\ast }
}\]
where we recall that $p\colon [n]\longrightarrow [0]$ is the unique map, and $\{B^{\phi,\psi}\}$ is a family for $f$ given by the damping condition. The functor $B^{\eta\circ p,\eta_0\circ p}_\ast$ sends $(p,\underline{a}\leq 0;G)$ to $(p,\underline{a}\leq 0;B_\ast(G))$ with
\[B_\ast(G)=s_{\underline{a}}^{\ast}(G|_{(-\infty,\epsilon)}\cup B^{\eta\circ p,\eta_0\circ p}|_{(-\epsilon,\infty)})\]
where $s_{\underline{a}}\colon\mathbb{R}\longrightarrow\mathbb{R}$ is a diffeomorphism that sends $0$ to $1$ with the same properties as in the proof of \ref{secondkindzero}. The diagram does not commute strictly, but there is a homotopy between $(\eta_0)_\ast\circ B^{\eta\circ p,\eta_0\circ p}_\ast$ defined in a completely analogous way as the homotopy of the proof of \ref{secondkindzero}. The map $\eta_0$ is a homotopy equivalence by \ref{secondkindzero}, and $B^{\eta\circ p,\eta_0\circ p}_\ast$ is the scaling of a map of the first kind. Choosing right and left inverses for $B^{\eta\circ p,\eta_0\circ p}$ defines left and right homotopy inverses for $B^{\eta\circ p,\eta_0\circ p}_\ast$ analogously as in the proof of \ref{firstkind} above.
\end{proof}

Given $(b,g)\in \Ob\partial\C$, let $F_{(b,g)}^t$ be the topological poset with the same underlying set as $F_{(b,g)}$ but topologized using the standard topology of $\mathbb{R}$ on the real coordinate.

\begin{lemma}\label{topdiscronfib}
Suppose that orthogonal elements of $\Psi(M)$ are almost open.
The map of topological posets $F_{(b,g)}\longrightarrow F_{(b,g)}^t$ induced by the inclusion of $\mathbb{R}$ with the discrete topology into $\mathbb{R}$ with the standard topology is a weak equivalence on classifying spaces
\[BF_{(b,g)}\stackrel{\simeq}{\longrightarrow} BF_{(b,g)}^t\]
\end{lemma}

\begin{proof}
Again assume for simplicity that $b=0$.
The projection off the real coordinate induces a commutative diagram
\[\xymatrix{BF_{(0,g)}\ar[dr]_{\simeq}\ar[r]& BF_{(0,g)}^t\ar[d]\\
&\Psi(M;g)
}\]
and the diagonal map is an equivalence by \ref{htpyfiber}. Therefore it is enough to show that the projection 
\[BF_{(0,g)}^t\longrightarrow\Psi(M;g)\]
is a weak equivalence.
Let $\mathcal{E}_g$ be the topological poset
\[\mathcal{E}_g=\coprod_{a\leq 0}U_{a}^{g}\]
where
\[U_{a}^{g}:=U_a\cap\Psi(M;g)=\colim_{\epsilon>0}\{G\in U_a:\ G|_{(-\epsilon,\infty)}=g|_{(-\epsilon,\infty)}\}\]
Here the real coordinate $a\leq 0$ has the standard topology, and $U_a$ is the open subset of $\Psi(M)$ weakly equivalent to $\Psi(M)^{\bot_a}$ given by the almost open condition (see \ref{almostopen}). Notice that since $U_a$ is open in $\Psi(M)$ the colimit topology on $U_{a}^{g}$ is the same as the topology as a subspace of $\Psi(M;g)$. The poset structure is defined as usual by $(a,G)\leq (a',H)$ if $a\leq a'$ and $G=H$.
Since the maps $\Psi(M)^{\bot_a}\longrightarrow U_a$ are weak equivalences, the map above factors as
\[BF_{(0,g)}^t\stackrel{\simeq}{\longrightarrow} B\mathcal{E}_g\stackrel{\pr}{\longrightarrow}\Psi(M;g)\]
We define an inverse in homotopy groups for $\pr\colon B\mathcal{E}_g\longrightarrow\Psi(M;g)$. Let $f\colon S^{n}\longrightarrow \Psi(M;g)$ be a continuous map representing an element of $\pi_n\Psi(M;g)$. Since the family $\{U_{a}\}$ covers $\Psi(M)$, for every $x\in S^n$ there is a $a_x\leq 0$ with $f(x)\in U_{a_x}^{g}$. Since every $U_{a_x}^{g}$ is open in $\Psi(M;g)$, the family $V_{a_x}:=f^{-1}(U_{a_x})$ is an open cover of $S^n$. Pick a finite sequence $a_0<\dots< a_k$ such that $V_{a_0},\dots,V_{a_k}$ covers the sphere, and choose a partition of unity $\phi_{0},\dots,\phi_k$ subordinate to the covering $\{V_{a_i}\}$. Such a choice of coordinates for every representative $f\colon S^{n}\longrightarrow \Psi(M;g)$ of homotopy class defines a map
\[s\colon\pi_n \Psi(M;g)\longrightarrow\pi_nB\mathcal{E}_g\] given by
\[s(f)(x)=[(a_0\leq\dots\leq a_k;f(x)),(\phi_{0}(x),\dots,\phi_k(x))]\]
where $(\phi_{0}(x),\dots,\phi_k(x))\in\Delta^k$. The map $s$ is a section for the projection map $\pi_nB\mathcal{E}_g\longrightarrow \pi_n \Psi(M;g)$.

Given a map $h\colon S^{n}\longrightarrow B\mathcal{E}_g$ We define a homotopy between $s(\pr\circ h)$ and $h$. At every point $x\in S^n$ the element $h(x)$ is an equivalence class of the form
\[h(x)=[(b_0(x)\leq\dots\leq b_{m_x}(x)\leq 0;G_x),(t_0(x),\dots,t_{m_x}(x))]\]
for a $G_x\in U_{b_0(x)}^g\cap\dots\cap U_{b_{m_x}(x)}^g\subset\Psi(M;g)$ and a $\underline{t}(x)=(t_0(x),\dots,t_{m_x}(x))\in\Delta^{m_x}$. Notice that we can assume $b_i(x)\neq b_j(x)$ for $i\neq j$. Let $\underline{a}=(a_1<\dots< a_k)$ be the sequence chosen for the map $\pr\circ h$ for the definition of $s(\pr\circ h)$, and $\underline{\phi}(x)=(\phi_1(x),\dots,\phi_k(x))$ the associated partition of unity. Denote $\underline{c}(x)$ the ordered union of the sequences $\underline{b}(x)=(b_0(x)\leq\dots\leq b_{m_x}(x))$ and $\underline{a}$. Let $\omega_x\colon [k]\longrightarrow[m_x+k+1]$ be the unique map such that $\omega_{x}^{\ast}\underline{c}(x)=\underline{a}$, and $\rho_x\colon [m_x]\longrightarrow[m_x+k+1]$ the one satisfying $\rho_{x}^{\ast}\underline{c}(x)=\underline{b}(x)$. Define a map $H\colon S^n\times[0,1]\longrightarrow B\mathcal{E}_g$ by
\[H(x,u)=[(\underline{c}(x);G_x),u\cdot(\omega_x)_\ast\underline{\phi}(x)+(1-u)\cdot(\rho_x)_\ast\underline{t}(x)]\]
This is a homotopy from $h$ to $s(\pr\circ h)$.
\end{proof}

\begin{rem}
This last lemma essentially shows that we can endow all our constructions with either the standard or the discrete topology on the real coordinate without changing the weak homotopy type. However, it is easier to deal with the hypothesis of Quillen's theorem B if one uses the discrete topology. 
\end{rem}


\section{The Madsen-Weiss theorem as a relative h-principle}\label{4aut}\label{sec6}

Let $\Psi_k\colon\emb_d\longrightarrow \Top$ be the sheaf of \cite{osmonoids}. It associates to a manifold $M$ the space $\Psi_k(M)$ of $k$-dimensional submanifolds of $M$ that are closed as a subset.
This space is topologized by defining a neighborhood $\mathcal{V}_{K,W}(G)$ of $G\in\Psi_k(M)$ for each pair $(K,W)$, where $K\subset M$ is compact and $W\subset\emb(N,M)$ is a neighborhood of the inclusion $N\subset M$ for the $C^\infty$-topology. The neighborhood is defined by
\[\mathcal{V}_{K,W}(N)=\{P\in\Psi_k(M)|P\cap K=j(N)\cap K\mbox{ for some }j\in W\}\]
The functor $\Psi_k$ sends an embedding $e\colon M\longrightarrow M'$ to
\[e^{\ast}(G):=e^{-1}(G)\]
The sheaf gluing property is given by union of subsets.

The cobordism category of $k$-dimensional bordisms in $\mathbb{R}^d$ of \cite{GMTW} is the topological category $\mathcal{C}_{k,d}$ with objects space
\[\coprod_{a\in\mathbb{R}}\{g\in\Psi_k(\mathbb{R}^{d-1})| g \mbox{ is compact}\}\]
and morphisms space
\[\coprod_{a\leq b\in\mathbb{R}}\colim_{\epsilon>0}\{G\in \Psi_k(\mathbb{R}^{d-1}\times[a,b])| G\bot_{[a,a+\epsilon)}, G\bot_{(b-\epsilon,b]}, G \mbox{ is compact}\}\]
where we denoted $G\bot_{J}$ if $G\cap\mathbb{R}^{d-1}\times J=N\times J$ for some $(k-1)$-submanifold $N\subset \mathbb{R}^{d-1}$. Source and target take intersection with $\mathbb{R}^{d-1}\times a$ and $\mathbb{R}^{d-1}\times b$ respectively, and composition of morphisms is defined by union. Notice that if we drop the compactness conditions this category would be isomorphic to $\partial\C(\mathbb{R}^{d-1}\times(-\infty,0])$ of §\ref{sec2}.

In this section we use theorem \ref{main} give a formal proof of the following theorem. Let $\mathcal{G}_{k,d}$ be the Grassmanians of $k$-planes in $\mathbb{R}^{d}$, and 
\[\gamma_{k,d}^\bot=\{(V,v)\in\mathcal{G}_{k,d}\times\mathbb{R}^{d}|v\bot V\}\]
 the complement of the tautological bundle over $\mathcal{G}_{k,d}$.
\begin{theorem}[\cite{MadsenWeiss},\cite{GMTW}]
There is a weak equivalence
\[B\mathcal{C}_{k,d}\simeq\Omega^{d-1}Th(\gamma_{k,d}^\bot)\]
where $Th(\gamma_{k,d}^\bot)$ denotes the Thom space.
\end{theorem}
 
\begin{proof}
Following \cite{david}, we prove the theorem using the relative h-principle map $\Psi_k(M;g_0)\longrightarrow\Psi_{k}^\ast(M;h(g_{0}))$ for $M=D^{d-1}\times\mathbb{R}$ and boundary condition $g_0=\emptyset\in\Psi_k(D^{d-1}\times\mathbb{R})$. By expanding the interior of $D^{d-1}$ to $\mathbb{R}^{d-1}$, we see that $\Psi_k(D^{d-1}\times\mathbb{R};\emptyset)$ is homeomorphic to the space of $k$-submanifolds of $\mathbb{R}^d$ which are bounded in the first $d-1$ components. This is the space $\mathcal{D}_{k,(1,d)}$ of \cite{david}, which is weakly equivalent to the classifying space of $\mathcal{C}_{k,d}$ via a zig zag
\[B\mathcal{C}_{k,d}\stackrel{\simeq}{\longleftarrow}BD\stackrel{\simeq}{\longrightarrow}\mathcal{D}_{k,(1,d)}\cong \Psi_k(D^{d-1}\times\mathbb{R};\empty)\]
through a topological poset $D$. This construction is analogous to our proof of $\ref{cd}$ an we refer to \cite{david} for the details.

Recall that $\Psi_{k}^\ast(D^{d-1}\times\mathbb{R})$ is defined in terms of sections of a certain bundle over $D^{d-1}\times\mathbb{R}$. Since $D^{d-1}\times\mathbb{R}$ is canonically contractible this bundle is trivial, and the space $\Psi_{k}^\ast(D^{d-1}\times\mathbb{R})$ is canonically homeomorphic to the space of maps $D^{d-1}\times\mathbb{R}\longrightarrow\Psi_k(\mathbb{R}^{d})$. Under this identification the boundary condition $h(\emptyset)$ is a constant map, and contracting the $\mathbb{R}$ component gives a homotopy equivalence \[\Psi_{k}^\ast(D^{d-1}\times\mathbb{R};h(\emptyset))\simeq\Omega^{d-1}\Psi_{k}(\mathbb{R}^d)\]
The map $Th(\gamma_{k,d}^\bot)\longrightarrow\Psi_{k}(\mathbb{R}^d)$ that sends a pair $(V,v)$ to $V+v\subset\mathbb{R}^d$, and the basepoint $\infty$ to $\emptyset$ is a weak equivalence (see \cite{david}). All together we get a diagram
\[B\mathcal{C}_{k,d}\simeq\Psi_k(D^{d-1}\times\mathbb{R};\emptyset)\stackrel{h}{\longrightarrow}\Psi_{k}^\ast(D^{d-1}\times\mathbb{R};h(\emptyset))\simeq\Omega^{d-1}Th(\gamma_{k,d}^\bot)\]
We prove in \ref{cond1},\ref{transalmostopen}, \ref{psikdamping} and \ref{mangplike} below that the sheaf $\Psi_k$ satisfies the conditions of \ref{main} for $M=D^{d-1}\times\mathbb{R}$, and therefore the relative h-principle map
\[\Psi_k(D^{d-1}\times\mathbb{R};\emptyset)\stackrel{h}{\longrightarrow}\Psi_{k}^\ast(D^{d-1}\times\mathbb{R};h(\emptyset))\]
is a weak equivalence.
\end{proof}

\begin{rem}
Our main theorem gives a relative h-principle for any boundary condition $g\in\Psi_k(\partial (D^{d-1}\times\mathbb{R}))=\Psi_k((S^{d-2}\times\mathbb{R})\times\mathbb{R})^{\bot_{\mathbb{R}}}$ that is not necessarily the empty manifold. One can define a category $\mathcal{C}_{k,d}^g$ whose morphisms are bordisms in $D^{d-1}\times[a,b]$ that agree with $g$ in a neighborhood of $(\partial D^{d-1})\times[a,b]$. Our proof gives an equivalence
\[B\mathcal{C}_{k,d}^g\simeq\map_{h(g)}(D^{d-1}\times\mathbb{R},\Psi_k(\mathbb{R}^d))\]
where $\map_{h(g)}(D^{d-1}\times\mathbb{R},\Psi_k(\mathbb{R}^d))$ is the space of maps $D^{d-1}\times\mathbb{R}\longrightarrow \Psi_k(\mathbb{R}^d)$ that agree with $h(g)\colon (S^{d-2}\times\mathbb{R})\times\mathbb{R}\longrightarrow \Psi_k(\mathbb{R}^d)$ near the boundary. Here $h(g)$ is the scanning of the boundary condition $g$, and it is in general not a constant map.
\end{rem}

\begin{rem}
If $\theta\colon B\longrightarrow BGL_d$ is a fibration, a tangential structure on a $d$-manifold $W\subset\mathbb{R}^k$ is a lift along $\theta$ for the classifying map $W\longrightarrow \mathcal{G}_{k,d}$ of the tangent bundle. There are analogues  $\Psi_{k}^\theta$ and $\mathcal{C}_{k,d}^{\theta}$ where the manifolds are equipped with a tangential structure, and a Madsen-Weiss theorem
\[B\mathcal{C}_{k,d}^{\theta}\simeq \Omega^{d-1}Th(\theta^\ast\gamma_{k,d}^\bot)\]
(see \cite{GMTW}, \cite{soren}, \cite{david}, \cite{osmonoids}). All the steps of our proof go through in this more general situation, except the group-like condition that might fail for a general $\theta$. It does however hold in the case of oriented manifolds, when $\theta\colon \mathcal{G}_{k,d}^+\longrightarrow\mathcal{G}_{k,d}$ is the projection map from the oriented Grassmanians.
\end{rem}

Condition (1) of \ref{main} holds for $\Psi_k$ for deep results of \cite{oscar} and \cite{gromov}.
\begin{prop}[\cite{oscar},\cite{gromov}]\label{cond1}
If $M$ is an open manifold, the h-principle maps $\Psi_k(M)\longrightarrow\Psi_{k}^\ast(M)$ and $\Psi_{k}(\partial M)\longrightarrow\Psi_{k}^\ast(\partial M)$ are weak equivalences.
\end{prop}

\begin{proof}
In \cite{oscar} the author proves that the sheaf 
\[\mathcal{O}^{op}(W)\longrightarrow\Top\]
defined on the category of open subsets of any $d$-manifold $W$ by restricting $\Psi_k$ is microflexible. By Gromov's theorem \cite{gromov} the h-principle maps $\Psi_k(M)\longrightarrow\Psi_{k}^\ast(M)$ and $\Psi_{k}(\partial M)\longrightarrow\Psi_{k}^\ast(\partial M)$ are weak equivalences for $M$ open.
\end{proof}

Notice that for $M=D^{d-1}\times\mathbb{R}$ the space $\Psi^{\ast}(D^{d-1}\times\mathbb{R})$ is canonically equivalent to $\Psi(\mathbb{R}^d)$ since $D^{d-1}\times\mathbb{R}$ is contractible, and the h-principle map $\Psi(D^{d-1}\times\mathbb{R})\longrightarrow \Psi^{\ast}(D^{d-1}\times\mathbb{R})$ is always an equivalence (this is however not true in general for $\partial(D^{d-1}\times\mathbb{R})$).

The main ingredient of the following is a result of \cite{GMTW}, where the authors show that the inclusion of orthogonal elements into transverse elements is a weak equivalence.
\begin{prop}[\cite{GMTW}]\label{transalmostopen}
The orthogonal elements of the sheaf $\Psi_k$ are almost open at $M$, if $\partial M=N\times\mathbb{R}^n$ for some compact manifold $N$ and $0\leq n\leq d$. 
\end{prop}

\begin{proof}
Given a real number $a$, define $\Psi_k(M)^{\pitchfork a}\subset \Psi_k(M)$ to be the subspace of manifolds that intersect $\partial M\times a$ transversally. Clearly every submanifold that intersects $\partial M\times a$ orthogonally belongs to $\Psi_k(M)^{\pitchfork a}$, and this defines a map
\[\Psi_k(M)^{\bot_a}\longrightarrow \Psi_k(M)^{\pitchfork a}\]
This is showed to be a weak equivalence in \cite{david} and \cite{osmonoids}. Unfortunately, the space $\Psi_k(M)^{\pitchfork a}$ is not open in $\Psi_k(M)$ unless $M$ is compact. If $\partial M=N\times\mathbb{R}^n$, let $B^{\delta}\subset\mathbb{R}^n$ be the open ball of radius $1+\delta$ centered at the origin. Define $U_a$ as the subspace of $\Psi_k(M)$ of submanifolds that are transverse to $\partial M\times a$ in a neighborhood of the unit compact ball
\[U_a:=\{G\in \Psi_k(M)| (G\cap N\times B^{\delta})\pitchfork \partial M\times a \mbox{ for some }\delta>0\}\]
Clearly $\Psi_k(M)^{\pitchfork a}\subset U_a$. The subspace $U_a$ is open, since given a $G\in U_a$ one can find a neighborhood of the inclusion $W\subset\emb(G,M)$ small enough so that $\mathcal{V}_{K,N\times K}\subset\Psi_k(M)^{\pitchfork a}$ for every compact subset $K\subset\mathbb{R}^n$ containing the compact unit ball.

It remains to show that the inclusion $\Psi_k(M)^{\pitchfork a}\longrightarrow U_a$ is a weak equivalence. Given a map $f\colon S^{n}\longrightarrow U_a$ one can choose $\delta>0$ such that $f(x)\cap N\times B^{\delta}$ is transverse to $\partial M\times a$ for all $x$. Pick a diffeomorphism $\phi\colon \mathbb{R}^n\longrightarrow B^{\delta}$, and define a map $r(f)\colon S^n\longrightarrow \Psi_k(M)^{\pitchfork a}$ by
\[r(f)(x)=(\id_N\times\phi)^{\ast} f(x)\]
This construction defines an inverse for the inclusion map on the $n$-th homotopy group.
\end{proof}

The main ingredient for the next result is the smooth approximation theorem for $\Psi_k$ of \cite{osmonoids}.

\begin{prop}\label{psikdamping}
The sheaf $\Psi_k$ is damping at any manifold $M$.
\end{prop}

\begin{proof}
For any $d$-manifold $W$, define a map $F\colon [0,1]\longrightarrow \Psi_k(W)$ to be smooth if the subset
\[E=\{(p,t)\in W\times[0,1]|p\in F(t)\}\]
is a submanifold of $W\times [0,1]$.
By a result of \cite{osmonoids}, one can homotope compact families of continuous maps $[0,1]\longrightarrow \Psi_k(W)$ to compact families of smooth maps keeping fixed subintervals of $[0,1]$ where the family was already smooth, just as for usual manifolds.

Let $f\colon\Delta^q\longrightarrow\Psi_k(\partial M)^{\bot_{\mathbb{R}}}$ be a continuous map.
Since the values of $f$ are orthogonal on $\mathbb{R}$, the submanifolds $f(x)\subset\partial M\times (0,\infty)$ are of the form
\[f(x)=N_x\times (0,\infty)\]
for some $(k-1)$-dimensional submanifold $N_x$ of $\partial M$. This defines a continuous map $\overline{f}\colon\Delta^q\longrightarrow\Psi_{k-1}(\partial M)$ by $\overline{f}(x)=N_x$. Here we consider $\Psi_{k-1}$ as a sheaf on $(d-1)$-dimensional manifolds
\[\Psi_{k-1}\colon\emb_{d-1}\longrightarrow\Top\]
so that $\Psi_{k-1}(\partial M)$ consists of actual submanifolds of $\partial M$ and not of $\partial M\times\mathbb{R}$.
Let $\alpha\colon[0,1]\longrightarrow[0,1]$ be a monotone smooth map with constant value $0$ in a neighborhood of $0$, with constant value $1$ in a neighborhood of $1$, and that agrees with the identity away from zero and one. Consider the continuous map $H\colon \Delta^{q}\times\Delta^{q}\times[0,1]\longrightarrow \Psi_{k-1}(\partial M)$ defined by
\[H_{x,y}(t)=\overline{f}(x\cdot (1-\alpha(t))+y\cdot \alpha(t))\]
For every fixed $x,y\in\Delta^q$, the map
\[H_{x,y}\colon [0,1]\longrightarrow \Psi_{k-1}(\partial M)\]
is smooth near $0$ and $1$ by our choice of $\alpha$. Approximate $H$ to a map $K\colon \Delta^{q}\times\Delta^{q}\times[0,1]\longrightarrow \Psi_{k-1}(\partial M)$ such that $K_{x,y}\colon [0,1]\longrightarrow \Psi_{k-1}(\partial M)$ is smooth for all $x,y\in\Delta^q$ and that agrees with $H_{x,y}$ around $0$ and $1$.
Since $K_{x,y}$ is smooth, the space
\[B_{x,y}:=\{(p,t)\in \partial M\times[0,1]|p\in K_{x,y}(t)\}\]
is a smooth $k$-dimensional submanifold of $\partial M\times [0,1]$. Since $H_{x,y}$ is constant near $0$ and $1$, the submanifold $B_{x,y}$ agrees with $f(x)$ near $\partial M\times 0$ and with $f(y)$ near $\partial M\times 1$. Given $\phi,\psi\colon[n]\longrightarrow[q]$ define a continuous map $B^{\phi,\psi}\colon\Delta^n\longrightarrow \Psi_k(\partial M)$ for the damping condition by
\[B^{\phi,\psi}(x)=B_{\phi_\ast x,\psi_\ast x}\]
as a submanifold of $\partial M\times\mathbb{R}$ extended orthogonally outside $\partial M\times [0,1]$.
\end{proof}

\begin{prop}\label{mangplike} If $\partial M=N\times\mathbb{R}$ for some $(d-2)$-manifold $N$ without boundary, $\Psi_k$ is group like at $M$.
\end{prop}

\begin{proof}
Let us define an inverse for an element $G\in\Psi_k(\partial M)^{\bot_{(-\infty,a+\epsilon)}}\cap\Psi_k(\partial M)^{\bot_{(b-\epsilon,\infty)}}$. This is the same as a submanifold of $\partial M\times[a,b]$ which intersect the boundary components orthogonally.
Define \[\overline{G}=(2b-\id)^\ast G\] to be the flip of $G$ along $\partial M\times b$. Since $G$ intersect the boundary orthogonally $\overline{G}|_{(b-\epsilon,b+\epsilon)}=G|_{(b-\epsilon,b+\epsilon)}$. By abuse of notation we write
\[G\cup\overline{G}:=G|_{(-\infty,b+\epsilon)}\cup\overline{G}|_{(b-\epsilon,\infty)}\]
Let us define the path $\gamma\colon I\longrightarrow\Psi_k(\partial M)$ from $G\cup\overline{G}$ to a an element orthogonal over $\mathbb{R}$.
Recall that we are assuming $\partial M=N\times\mathbb{R}$. For every $t\in[0,1)$ let us denote $\phi_t\colon N\times\mathbb{R}\times\mathbb{R}$ the diffeomorphism defined by
\[\phi_t(x,u,v):=(x,u+\frac{t}{1-t},v)\]
The topology on $\Psi_k$ is defined so that the path $\omega\colon [0,1]\longrightarrow\Psi_k(\partial M)$ defined for $0\leq t<1$ by
\[\omega(t)=\phi_{t}^{\ast}(G\cup\overline{G})\]
and $\omega(1)=\emptyset$ is continuous. Similarly, the path $\omega'\colon [0,1]\longrightarrow\Psi_k(\partial M)$ defined for $0\leq t<1$ by
\[\omega(t)=\phi_{t}^{\ast}(\alpha^{\ast}G|_{(a-\epsilon,a+\epsilon)})\]
and by $\omega'(1)=\emptyset$ is continuous as well. Here $\alpha\colon \mathbb{R}\longrightarrow (a-\epsilon,a+\epsilon)$ is an orientation preserving diffeomorphism. The concatenation $\gamma'$ of $\omega$ with the inverse path of $\omega'$ is a path between $G\cup\overline{G}$ and the manifold $\alpha^{\ast}G|_{(a-\epsilon,a+\epsilon)}$ which is orthogonal on $\mathbb{R}$. The path $\gamma'$ is however not constant on a neighborhood of $(-\infty,a]\cup[b,\infty)$, but notice that $\gamma'(0)$ and $\gamma'(1)$ agree on $(-\infty,a+\epsilon)\cup(b-\epsilon,\infty)$. In order to obtain a path constant outside $[a,b]$ replace the family of diffeomorphisms $\id\times\phi_t\colon\mathbb{R}^2\longrightarrow\mathbb{R}^2$ with a family of diffeomorphisms $\psi_t\colon\mathbb{R}^2\longrightarrow\mathbb{R}^2$ with
\begin{itemize}
\item $\psi_t=\phi_t$ on $\mathbb{R}\times (a+\frac{2\epsilon}{3},b-\frac{2\epsilon}{3})$
\item $\psi_t=\id$ on $(-\infty,a+\frac{\epsilon}{3})\cup(b-\frac{\epsilon}{3},\infty)$
\item $\psi_0=\id$
\end{itemize}

A similar construction gives the path $\delta$ for $t_{ab}^{\ast}(\overline{G}|_{(-\infty,2a-b+\epsilon)})\cup G|_{(a-\epsilon,\infty)}$.
\end{proof}

\end{document}